\documentclass[12pt]{amsart}
\usepackage{setspace}
\usepackage{amsmath, amssymb,amsthm, amsfonts, latexsym}

\newtheorem{thm}{Theorem}[section]
\newtheorem{lemma}{Lemma}[section]
\newtheorem{prop}{Proposition}[section]
\newtheorem{cor}{Corollary}[section]

\newtheorem{ques}{Question}[section]

\theoremstyle{remark}
\newtheorem{remark}{Remark}[section]

\newcommand{\Ric}{\mbox{Ric}}
\newcommand{\R}{\mathbb R}

\numberwithin{equation}{section}
\newcommand{\be}{\begin{equation}}
\newcommand{\ee}{\end{equation}}
\def\p{\partial}

\def\la{\langle}
\def\ra{\rangle}
\def\lf{\left}
\def\ri{\right}
\def\Pi{\displaystyle{\mathbb{II}}}
\def\Ric{\text{\rm Ric}}
\def\e{\epsilon}

\def\vh{\vspace{.2cm}}

\def\mS{\mathbb{S}}

\def\bee{\begin{equation*}}
\def\eee{\end{equation*}}

\def\dels{\Delta_{_\Sigma}}

\def\nabs{\nabla_{_\Sigma} }
\def\lam{\lambda}
\def\Sel{\Sigma_{\eta, \lambda}}

\def\H{\mathbb{H}}
\def\mS{\mathbb{S}}
\def\vecHm{\vec{H}_{\mathbb{M}}}
\def\Sigmao{\Sigma_0} 

\def\mHs{\vec{H}_{\mS}}
\def\Pis{\Pi_{\mS}}
\def\onab{\overline{\nabla}}

\begin{document}

\title[Boundary effect of Ricci curvature]{Boundary effect of Ricci curvature}

\author{Pengzi Miao$^*$}
\address[Pengzi Miao]{Department of Mathematics, University of Miami, Coral Gables, FL 33146, USA.}
\email{pengzim@math.miami.edu}

\author{Xiaodong Wang}
\address[Xiaodong Wang]{Department of Mathematics, Michigan State University, East Lansing, MI 48864, USA.}
\email{xwang@math.msu.edu}

\thanks{$^*$Research partially supported by Simons Foundation Collaboration Grant for Mathematicians \#281105.}

\renewcommand{\subjclassname}{
  \textup{2010} Mathematics Subject Classification}
\date{}

\begin{abstract}
On a compact Riemannian manifold with boundary,  we study how Ricci curvature of the interior 
affects the geometry of the boundary. 
First we establish    integral  inequalities  for functions 
defined solely on the boundary and  apply them to obtain geometric inequalities 
involving  the total mean curvature. 
Then we  discuss  related rigidity questions and prove 
Ricci curvature rigidity results for manifolds with boundary.

\end{abstract}

\maketitle

\section{Introduction and statement of results}

In this paper, we consider the question  how the Ricci curvature of a compact manifold with boundary  affects 
the  boundary  geometry of the manifold. For the scalar curvature the same question is 
related to the quasi-local mass problem in general relativity. Indeed, much of the formulation of  the results 
in this paper is motivated by that  in \cite{ShiTam02, WangYau07, MiaoTam13}. 

We begin with  integral inequalities that hold for  functions solely defined on the boundary.
For simplicity, all manifolds and functions in this paper are assumed to be smooth. 

\begin{thm}  \label{thm-poincare-type-ineq-intro}
Let $(\Omega, g)$ be an $n$-dimensional, compact  Riemannian manifold  
   with nonempty  boundary $\Sigma$. Let $K $ be a constant that is 
   a lower bound of the Ricci curvature of $g$, i.e. 
   $ \Ric \ge K g . $
   Let $H$ be   the mean curvature of  $\Sigma$  in $(\Omega, g)$ with respect to the outward normal. 
   Suppose $H>0$. 
Given any function $\eta$ on $\Sigma$, define 
\bee
A(\eta) = \int_\Sigma \frac{\eta^2}{H} d \sigma , \ \ 
B(\eta) = \int_\Sigma \frac{ \eta \dels \eta }{H} d \sigma ,
\eee
\bee
C(\eta) = \int_\Sigma \lf[ \frac{ (\dels \eta)^2}{H} - \Pi (\nabs \eta, \nabs \eta) \ri] d \sigma ,
\eee
where $ \nabs$,  $\dels$ are the gradient,  the Laplacian on $\Sigma$ respectively,
$\Pi$ is the second fundamental form of $\Sigma$ and $d \sigma$ is the volume form on $\Sigma$. 
Then,  for each nontrivial $\eta$, either 
\be \label{eq-intro-case-1}
\lf(\frac{B(\eta)}{A(\eta)} \ri)^2 \le \frac{C(\eta)}{A(\eta) } 
\ee
or 
\be \label{eq-intro-case-2}
\frac12 K  \le -  \frac{  B(\eta)}{A(\eta)}  - \sqrt{ \lf(\frac{B(\eta)}{A(\eta)} \ri)^2 - \frac{C(\eta)}{A(\eta) }  }  . 
\ee
\end{thm}

\vspace{0.1cm}

\begin{remark}
If the term $\Pi (\nabs \eta , \nabs \eta)$ were absent in $C(\eta)$, 
then \eqref{eq-intro-case-1} would always hold  
by H\"{o}lder inequality. 
\end{remark}

\begin{remark} 
When $\Omega $ is the closure of a bounded domain in $\R^3$, 
the functional $C(\eta) $, up to a constant multiple of $\frac{1}{8 \pi}$, is  the 2nd variation of the  
Wang-Yau quasi-local energy (\cite{WangYau-PRL, WangYau08}) at $\Sigma = \p \Omega$ 
in $ \R^{3,1}$,
where $ \R^{3,1}$ is the $4$-dimensional Minkowski spacetime. (See \cite{MiaoTam13, MiaoTamXie11} for details.) 
\end{remark}

\begin{remark} 
If $\Sigma$ has  a  component $\Sigma_0$ on which $\Pi > 0$, then  
\eqref{eq-intro-case-1} always fails for  an $\eta$ which  is  a non-constant 
eigenfunction on $\Sigma_0$ and zero elsewhere. 
In this case,     \eqref{eq-intro-case-2} yields estimates  on the first 
nonzero eigenvalue of $\Sigma_0$.
(See Corollary \ref{cor-eigenvalue-est} for details.) 
\end{remark}

The conclusion of   Theorem \ref{thm-poincare-type-ineq-intro} 
is easily  seen to be  equivalent to   a statement  
\be \label{eq-thm-bdry-ineq-intro}
 \int_\Sigma  \Pi (\nabs  \eta , \nabs  \eta ) d \sigma  
\le    \int_{ \Sigma  }
\frac{1}{H} \lf(  \dels \eta   +   t  \eta  \ri)^2 d \sigma  
\ee
for all constants $ t \le \frac12 K $.  
In Theorem \ref{thm-poincare-type-ineq} of Section \ref{sec-basic-ineq}, 
we prove a  more general  version of \eqref{eq-thm-bdry-ineq-intro} which  allows  $H \ge 0$.
Interpreted this way, Theorem \ref{thm-poincare-type-ineq-intro} and its generalization 
(Theorem \ref{thm-poincare-type-ineq})  have natural applications to  the total  mean curvature of the boundary. 

We first  state the case of nonnegative Ricci curvature. 

\begin{thm} \label{thm-immersion-0-intro}
Let $(\Omega, g)$ be an $n$-dimensional, compact  Riemannian manifold with nonnegative Ricci curvature,
  with connected  boundary $\Sigma$ which has nonnegative mean curvature $H$.
Let $X: \Sigma  \rightarrow \R^m $ be an isometric immersion of $\Sigma$ 
  into some Euclidean space $\R^m$   of dimension  $m\ge n$.  Then 
\begin{equation} 
\label{eq-immersion-0-ineq-intro} 
\int_{\Sigma} H \ d \sigma \le \int_{\Sigma'}  \frac{ | \vec{H}_0 |^2 }{H}   \ d {\sigma} ,
\end{equation}
where  $\vec{H}_0$ is  the mean curvature vector of the  immersion $X$, $ | \vec{H}_0 |$ is  
the length of $\vec{H}_0$, and 
$ \Sigma' = \{ x \in \Sigma \ | \ \vec{H}_0 (x) \neq 0 \}$.
 Moreover, if equality in (\ref{eq-immersion-0-ineq-intro}) 
holds, then 
\begin{itemize}
\item[a)]  $ H = | \vec{H}_0 |  $ identically  on $\Sigma$.
\item[b)]  $(\Omega, g)$ is flat and  $X(\Sigma)$  lies in an $n$-dimensional  plane in  $\R^m$.
\item[c)]   $(\Omega, g)$ is isometric to a domain in $\R^n$ if $X$ is an embedding. 
\end{itemize} 
\end{thm}


\begin{remark}
In light of the  Nash imbedding theorem \cite{Nash56},  the boundary $\Sigma$ 
always admits  an isometric immersion into some Euclidean space.
Therefore,  Theorem \ref{thm-immersion-0-intro} applies to
any compact Riemannian manifold with nonnegative Ricci curvature, with mean convex boundary
(i.e.  $H \ge 0$). 
One may compare Theorem \ref{thm-immersion-0-intro} with the result in \cite{ShiTam02}  in which 
a weaker curvature  condition $R \ge 0$ is assumed, where $R$ is the scalar curvature,  while a more stringent  
boundary condition is imposed.
 \end{remark}
 
\begin{remark}  \label{rmk-bdry-connected}
If $(\Omega, g)$ has  nonnegative Ricci curvature and  nonempty mean convex boundary,  
 it was shown  in \cite{Ichida81, Kasue83}  (also cf.  \cite{HangWang07}) that $ \p \Omega $ 
 has at most two components, and  $\p \Omega  $ 
 has  two components only if $(\Omega, g)$ is isometric to $N \times I$ for 
 a connected closed  manifold $N$ and an interval $I$.
This is  why  we  only consider connected boundary  in  Theorem \ref{thm-immersion-0-intro}.
\end{remark}

\begin{remark}
Theorem \ref{thm-immersion-0-intro} generalizes   \cite[Proposition 2]{HangWang07}, 
which proves  that  b) and  c) hold  under a pointwise  assumption  $ H \ge | \vec{H}_0 | $.
Indeed  the proof in \cite{HangWang07}  
can be easily adapted to prove  our Theorem \ref{thm-immersion-0-intro}. 
\end{remark}

\begin{remark}
 If $H> 0$ on $\Sigma$,   Theorem \ref{thm-immersion-0-intro} implies 
$ \int_\Sigma H d \sigma \le C  $
where  $C>0$ is a  constant depending only on the induced metric  on $\Sigma$
 and a positive lower bound  of $H$.  
\end{remark}

Next, we give an  analogous  result  for manifolds with  positive Ricci curvature.

\begin{thm} \label{thm-immersion-p0-intro}
Let $(\Omega, g)$ be an $n$-dimensional, compact  Riemannian manifold 
with positive Ricci curvature,
 with connected  boundary $\Sigma$ which has nonnegative mean curvature $H$.
Let  $k>0$ be a constant such that  
$$ \Ric \ge (n-1) k g . $$
Suppose  there exists an isometric immersion $X: \Sigma \rightarrow \mS^{m}_{k}  $,
 where   $\mS^m_k$ is the  sphere of  dimension $m \ge n$ with constant sectional curvature $k$. 
 Then 
\begin{equation} 
\label{eq-immersion-p0-ineq-intro} 
\int_{\Sigma} H \ d \sigma < \int_{\Sigma}  \frac{ | \vec{H}_{\mS} |^2  + \frac14  (n-1)^2 k   }{H}   \ d {\sigma} ,
\end{equation}
where  $\vec{H}_{\mS} $ is  the mean curvature vector of the immersion $X$.
 \end{thm}

 Like \eqref{eq-immersion-0-ineq-intro},   \eqref{eq-immersion-p0-ineq-intro} 
imposes  constraints  on  the  boundary mean curvature 
when the Ricci curvature of the interior   has a positive lower bound.
For instance,  consider the standard hemisphere $(\mS^n_+, g_{_\mS})$ of dimension $n$.
Let   $ \Omega \subset \mS^n_+ $ be a smooth domain with connected  boundary. 
It follows from Theorem  \ref{thm-immersion-p0-intro} that 
there does {\em not} exist a metric $g$ on $\Omega$  satisfying  $\Ric \ge (n-1)$,
$g |_{T \p \Omega} =  g_{_\mS} |_{T \p \Omega} $  and $H \ge \sqrt{ ({H}_{\mS})^2 + \frac14 (n-1)^2 }$,
where $H$ and ${H}_{\mS}$ are  the mean curvature    of $\p \Omega$ in
$(\Omega, g)$ and $(\Omega, g_{_{\mS}})$ respectively. 
This could  be compared  with the first step,  i.e. \cite[Theorem 4]{BMN11}, in 
the construction of  the counterexample to Min-Oo's Conjecture,
 in which a metric on $\mS^n_+$ is produced so that it satisfies 
 $R \ge n(n-1)$,   but   the mean curvature of 
 $ \p \mS^n_+$  is raised  to be everywhere positive.
One may also compare this with the Ricci curvature rigidity theorems 
in \cite{HangWang09}.

When a  manifold has negative Ricci curvature  somewhere, we have 

\begin{thm} \label{thm-immersion-n0-intro}
Let $(\Omega, g)$ be an $n$-dimensional, compact  Riemannian manifold 
 with  boundary $\Sigma$ which has  nonnegative mean curvature $H$.
 Let   $k>0$ be  a constant satisfying 
$$ \Ric \ge - (n-1) k g . $$
 Suppose $\Sigma$ has a component $\Sigma_0$  which admits  an isometric immersion 
 $$X = (t, x_1, \ldots, x_n) : \Sigma_0  \longrightarrow \H^{m}_{-k}  \subset \R^{m,1}, $$
 where  $ \R^{m,1}$ is the $(m+1)$-dimensional Minkowski spacetime  with $m \ge n$ and 
 $$ \H^m_{-k}  = \lf\{ (y_0, y_1, \ldots, y_m  ) \subset \R^{m, 1}  \ | \ - y_0^2 + \sum_{i=1}^{m} y_i^2 = - \frac{1}{k}, \ y_0 > 0  \ri\} .$$ 
Then 
\be \label{eq-immersion-n0-ineq-intro}
\begin{split}
& \ 
  \int_{\Sigma_0} H  d \sigma +   \int_{\Sigma_0} \Pi (\nabs t , \nabs t) d \sigma  \\
 <     & \  \int_{\Sigma'_0}  \frac{1}{H}   \lf\{ 
    | \vec{H}_{\H} |^2  -  \frac14 (n-1)^2 k     +  \lf[  \dels t   -  \frac12 (n-1) k  t   \ri]^2 \ri\} 
 d \sigma  ,
 \end{split} 
\ee 
where   $\vec{H}_{\H} $ is  the mean curvature vector of the immersion $X$  into $\H^m_{-k}$,
\bee \label{eq-p-of-big-integrand-intro}
  |\vec{H}_{\H} |^2  -  \frac14 (n-1)^2 k     +  \lf[ \dels t   -  \frac12 (n-1) k  t  \ri]^2  \ge 0  \ \ \mathrm{on} \ \Sigma_0,
  \eee
and  $\Sigma'_0 $ is the set consisting of  $x \in \Sigma_0 $ such that 
$$  
|\vec{H}_{\H} |^2 (x) -  \frac14 (n-1)^2 k   
  +  \lf[ \dels t   -  \frac12 (n-1) k  t   \ri]^2 (x) > 0  . $$ 
 \end{thm}

\begin{remark}
  The term $ \int_{\Sigma_0} \Pi (\nabs t, \nabs t) d \sigma $  in \eqref{eq-immersion-n0-ineq-intro} can be dropped 
 if  either $ \Pi \ge 0 $ or   $ X(\Sigma_0) \subset \H^m_{-k} \cap \{ t = t_0 \}$ for some constant $t_0$.
For instance,  
this is the case if  $\Sigma_0$ can be isometrically immersed in a sphere. 
\end{remark}

The fact that \eqref{eq-immersion-p0-ineq-intro} and \eqref{eq-immersion-n0-ineq-intro} are  strict inequalities
is   due to  the characterization of  equality case in Theorem \ref{thm-poincare-type-ineq}.
This leads naturally  to rigidity questions in the context of Theorem \ref{thm-immersion-p0-intro} and
  \ref{thm-immersion-n0-intro}.  We have the following two related  results.

\begin{thm} \label{thm-spherical-intro} 
Let $\left( \Omega ,g\right)  $  be an $n$-dimensional,  compact Riemannian manifold
with boundary $\Sigma $.  Suppose 
\begin{itemize}
\item $\Ric\geq\left(  n-1\right)  g$

\item there exists  an isometric immersion $X  :  \Sigma \rightarrow\mathbb{S}^{m}$, where $ \mS^m$ is a standard  sphere of 
dimension $m \ge n $

\item $\Pi\left(  v,v \right)  \geq\left\vert \Pi_\mS \left( v ,  v \right)  \right\vert $, for any $ v \in T\Sigma$.
Here $\Pi$ is the second fundamental form of $\Sigma$  in $(\Omega, g)$ and 
 $\Pi_\mS$ is the vector-valued, second fundamental form of  the immersion $X $. 
\end{itemize}
Then $\left(  \Omega ,g\right)  $ is spherical, i.e. having constant sectional curvature $1$.  
Moreover  if $\Sigma$ is simply connected, then $\left(  \Omega ,g\right)  $ is isometric to a domain in $\mathbb{S}_{+}^{n}$.
\end{thm}

\begin{thm} \label{thm-hyperbolic-intro}
Let $\left(  \Omega, g\right)  $ be an $n$-dimensional,  compact Riemannian manifold with
boundary $\Sigma$. Suppose

\begin{itemize}
\item $\Ric\geq-\left(  n-1\right)  g$

\item there exists an isometric immersion $X:   \Sigma  \rightarrow\mathbb{H}^{m}$, 
where $\mathbb{H}^{m}$ is a hyperbolic space of dimension $m\geq n$

\item $ \Pi\left(  v,v\right)  \geq \left\vert \Pi_{\mathbb{H}}\left(v,v\right)  \right\vert $, for any $v\in T\Sigma$. 
Here $\Pi$ is the second fundamental form of $\Sigma$  in $(\Omega, g)$ and 
$\Pi_{\mathbb{H}}$ is the vector-valued, second fundamental form of  the immersion $X $. 
\end{itemize}
Then $\left( \Omega,g\right)  $ is hyperbolic, i.e. having constant sectional curvature $-1$. 
Moreover if $\Sigma$ is simply  connected, then $\left(  \Omega ,g\right)  $ isometric to a domin in $\mathbb{H}^{n}$.
\end{thm}

The rest of this paper is organized as follows. In Section \ref{sec-basic-ineq}, we prove 
Theorem \ref{thm-poincare-type-ineq} which implies Theorem \ref{thm-poincare-type-ineq-intro}. 
In Section \ref{sec-apps-bdry-immersion}, we consider applications of Theorem  \ref{thm-poincare-type-ineq} to 
  the total boundary  mean curvature   and prove 
  Theorem \ref{thm-immersion-0-intro} -- \ref{thm-immersion-n0-intro}. 
  In Section \ref{rigidity}, we discuss  the related rigidity question  and prove Theorem \ref{thm-spherical-intro} and \ref{thm-hyperbolic-intro}.

\section{A geometric Poincar\'e type inequality} \label{sec-basic-ineq}

The main result of this section is the following  geometric Poincar\'e type inequality for functions
defined  on the boundary of a compact Riemannian manifold. 

\begin{thm}  \label{thm-poincare-type-ineq}
Let $(\Omega, g)$ be an $n$-dimensional, compact  Riemannian manifold  
   with nonempty  boundary $\Sigma$. 
   Suppose 
   $$ \Ric \ge (n-1) k g  \ \ {and} \ \ H \ge 0 , $$
   where $\Ric $ is the Ricci curvature of $g$, $k$ is some constant,  and $H$ is  
    the mean curvature of  $\Sigma$ in $(\Omega, g)$  with respect to the outward normal. 
Suppose $H$ is not identically zero.   
Then
\be \label{eq-thm-bdry-ineq}
 \int_\Sigma  \Pi (\nabs  \eta , \nabs  \eta ) d \sigma  
\le    \int_{ \Sigma  \setminus  \lf\{  \dels \eta   + t \eta   = 0 \ri\} }
\frac{1}{H} \lf(  \dels \eta   +   t  \eta  \ri)^2 d \sigma  
\ee
 for any  nontrivial function $\eta $ on $\Sigma$ and any constant $ t \le \frac12  (n-1)  k   $. 
   Here  $\Pi (\cdot, \cdot)$ is   the second fundamental 
form of $\Sigma$, 
$ \nabs$ and $\dels$ denote the gradient and the Laplacian on $\Sigma$ respectively.
Moreover, equality in \eqref{eq-thm-bdry-ineq}  holds only if  
either $ k > 0 $, $t =0 $ and $\eta $ is a constant; or 
$ k = t = 0 $ and 
$ \eta $ is the boundary value of some function $u$ on $\Omega$ satisfying $ \nabla^2 u = 0 $.
 Here  $ \nabla^2$  denotes the Hessian   on $(\Omega, g)$.
\end{thm}

\begin{remark}
The case  $ k = 0 $, $H> 0$ and $ t= 0 $ in  \eqref{eq-thm-bdry-ineq} was first proved in \cite{MiaoTamXie11} 
and is related to the second variation of Wang-Yau quasi-local energy   \cite{WangYau-PRL, WangYau08} 
at a closed $2$-surface in $ \R^3 \subset \R^{3,1}$. 
\end{remark}

\begin{proof}
The  basic tool we use is Reilly's formula (\cite{Reilly77})
\be \label{eq-Reilly}
\begin{split}
& \ \int_\Omega \lf[ | \nabla^2 u |^2 -   ( \Delta u )^2 + \Ric (\nabla u, \nabla u) \ri] d V  \\
= & \int_\Sigma \lf[  - \Pi (\nabs u, \nabs u) - 2 (\dels u) \frac{\p u}{\p \nu}
- H \lf( \frac{\p u}{\p \nu} \ri)^2 \ri] d \sigma  , 
\end{split}
\ee
 which follows from integrating   the Bochner  formula. 
Here  $ \Delta $, $dV$  denote  the Laplacian, the volume form  on $ (\Omega, g)$ respectively; 
 $ \nu$ is the outward unit normal to $\Sigma$, and  $u$ is any function defined on $\Omega$. 

Given any  nontrivial $\eta $ on $\Sigma$ and any  constant $ \lam \le n k  $, let  $u$ be the unique solution 
to 
\be \label{eq-u-extension}
\left\{ 
\begin{array}{rcl}
\Delta u + \lam u & = & 0 \ \ \ \mathrm{on} \ \Omega \\
u & = & \eta \ \ \ \mathrm{at} \ \Sigma  .
\end{array}
\right.  
\ee
The fact that \eqref{eq-u-extension} has a unique solution  in the case $k > 0 $ follows 
from another   theorem of Reilly (\cite[Theorem 4]{Reilly77})
which states that   the first Dirichlet eigenvalue $\lam_1 $ of $\Delta$  satisfies $\lam_1 \ge n k $
and $\lam_1 = n k $ if and only if $(\Omega, g)$ is isometric to a  hemisphere in which case 
$\Pi$ is  identically   zero. 
Plug this $u$ in \eqref{eq-Reilly},  using  the fact
\begin{align}
 | \nabla^2 u |^2 &  =    \frac{1}{n} ( \Delta u)^2  +  | \nabla^2 u - \frac{1}{n} (\Delta u) g |^2 ,   \nonumber \\
\lam \int_\Omega u^2 d V & =  \int_\Omega | \nabla u |^2 d V - \int_\Sigma u \frac{\p u}{\p \nu} d \sigma  \nonumber ,
\end{align}
and the assumption $ \Ric \ge (n-1) k g$, 
we have 
\be \label{eq-Reilly-b}
\begin{split}
& \  \lf( 1 - \frac{1}{n} \ri)  ( n k  -  \lam )  \int_\Omega | \nabla u |^2 d V + \int_\Omega | \nabla^2 u +  \frac{1 }{n} \lam u g |^2 d V \\
\le  &  \ 
 \int_\Sigma \lf[  - \Pi (\nabs \eta, \nabs \eta ) - 2 \lf( \dels \eta   + \frac{n-1}{2n}  \lam \eta  \ri)
\frac{\p u}{\p \nu} - H \lf( \frac{\p u}{\p \nu} \ri)^2 \ri]  d \sigma   . 
\end{split}
\ee
Given any  constant $ \epsilon > 0 $,  \eqref{eq-Reilly-b} implies 
\be \label{eq-Reilly-c}
\begin{split}
& \ \int_\Sigma \lf[   - \Pi (\nabs \eta, \nabs \eta )  +  \frac{1}{H+ \epsilon} {\lf(  \dels \eta   + \frac{n-1}{2n}  \lam \eta  \ri)^2 }
 + \epsilon  \lf( \frac{\p u}{\p \nu} \ri)^2  \ri] d \sigma  \\
\ge  & \   \int_\Sigma    \lf[ \frac{ 1 }{\sqrt{H + \epsilon } } \lf(  \dels \eta   + \frac{n-1}{2n}  \lam \eta  \ri) 
  +   \sqrt{H + \epsilon }   \frac{\p u}{\p \nu} \ri]^2  d \sigma \\
& \   +   \lf( 1 - \frac{1}{n} \ri)  ( n k  -  \lam )  \int_\Omega | \nabla u |^2 d V 
+ \int_\Omega | \nabla^2 u  +  \frac{1 }{n} \lam u g  |^2 d V \\
\ge & \ 0 .
\end{split}
\ee
Define 
$ \Sigma_{\eta, \lam}  = \lf\{ x \in \Sigma \ | \ \dels \eta   + \frac{n-1}{2n}  \lam \eta  = 0 \ri\} $. 
By  Lebesgue's monotone convergence theorem,  we have 
\be \label{eq-limit-Lebesgue}
\begin{split}
 & \ \lim_{ \epsilon \rightarrow 0}  \int_{ \Sigma  \setminus \Sigma_{\eta, \lam} }
\frac{1}{H+ \epsilon} \lf(  \dels \eta   + \frac{n-1}{2n}  \lam \eta  \ri)^2 d \sigma   \\
= & \  \int_{ \Sigma  \setminus  \Sigma_{\eta, \lam}   }
\frac{1}{H} \lf(  \dels \eta   + \frac{n-1}{2n}  \lam \eta  \ri)^2 d \sigma   .
\end{split}
\ee
Therefore, it follows  from \eqref{eq-Reilly-c} and  \eqref{eq-limit-Lebesgue}   that 
\be \label{eq-bdry-ineq-a}
 \int_\Sigma  \Pi (\nabs  \eta , \nabs  \eta ) d \sigma  
\le   \int_{ \Sigma  \setminus \Sigma_{\eta, \lam}   }
\frac{1}{H} \lf(  \dels \eta   + \frac{n-1}{2n}  \lam \eta  \ri)^2 d \sigma  ,
\ee
which proves  \eqref{eq-thm-bdry-ineq} by setting $t =  \frac{n-1}{2n}  \lam$. 

Next, suppose  
\be \label{eq-equality-1}
\begin{split}
 \int_\Sigma  \Pi (\nabs  \eta , \nabla_\Sigma \eta ) d \sigma 
 =  \int_{ \Sigma  \setminus \Sigma_{\eta, \lam}  } \frac{1}{H} \lf(  \dels \eta   + \frac{n-1}{2n}  \lam \eta  \ri)^2 d \sigma  .
\end{split} 
\ee
In particular, this shows 
\be
  \frac{1}{\sqrt{H}} \lf(   \dels \eta   + \frac{n-1}{2n}  \lam \eta \ri)  \in L^2 ( \Sigma \setminus \Sigma_{\eta, \lam}  )  
\ee
and the set 
$ \{ x \in \Sigma \setminus \Sigma_{\eta, \lam} \ | \ H (x) = 0 \} $
has $ d \sigma$-measure zero. 
Hence, by \eqref{eq-Reilly-b} and \eqref{eq-equality-1}, we have 
\be \label{eq-Reilly-d}
\begin{split}
& \  \lf( 1 - \frac{1}{n} \ri)  ( n k  -  \lam )  \int_\Omega | \nabla u |^2 d V   + \int_\Omega | \nabla^2 u  +  \frac{1 }{n} \lam u g  |^2 d V  \\
\le  &  \ 
 - \int_{\Sigma \setminus \Sel }   \lf[ \frac{ 1 }{\sqrt{H  } }  \lf(   \dels \eta   + \frac{n-1}{2n}  \lam \eta \ri) 
  +   \sqrt{H }   \frac{\p u}{\p \nu} \ri]^2  d \sigma  \\
  & \ -  \int_{\Sel} H  \lf( \frac{\p u}{\p \nu} \ri)^2  d \sigma ,
\end{split}
\ee
which implies 
\be \label{eq-interior-a}
  (nk - \lam) | \nabla u | = 0 ,  \ \ \nabla^2 u +  \frac{1}{n} \lam u  g = 0  \ \ \ \mathrm{on} \ \Omega
\ee
and
\be \label{eq-bdry-a}
     \dels \eta   +  {H }   \frac{\p u}{\p \nu} + \frac{n-1}{2n}  \lam \eta    = 0  \ \ \ \mathrm{at} \ \Sigma . 
\ee
 
If $ \lam < n k $,   \eqref{eq-interior-a}  shows $ u $ is identically a constant, therefore  $\lam = 0 $, $k > 0 $
and $\eta $ is a constant on $\Sigma$. 

If $ \lam = n k $,  \eqref{eq-interior-a}  shows
\be
\nabla^2 u + k u g = 0 \ \ \mathrm{on} \ \Omega  ,
\ee
which   implies 
\be \label{eq-bdry-b}
\dels \eta  + H \frac{\p u}{\p \nu} + (n-1) k \eta = 0  \ \ \mathrm{at} \ \Sigma. 
\ee
Comparing \eqref{eq-bdry-b} to \eqref{eq-bdry-a} with $\lam = nk$, we have $ k = \lam = 0 $.
This completes the proof. 
\end{proof}


When $H> 0 $, \eqref{eq-thm-bdry-ineq} is  simplified to
$$ \int_\Sigma  \Pi (\nabs  \eta , \nabs  \eta ) d \sigma  
\le    \int_{ \Sigma } \frac{1}{H} \lf(  \dels \eta   +   t  \eta  \ri)^2 d \sigma  .
 $$
In this case, Theorem \ref{thm-poincare-type-ineq}  is 
equivalent to a  statement that, given 
any nontrivial $\eta $ on $\Sigma$, the quadratic form 
\bee
Q_{\eta} (t) : =
 A (\eta) t^2 + 2 B(\eta) t + C(\eta) 
\eee
satisfies 
\be \label{eq-equiv-thm}
 Q_\eta (t) \ge 0 , \ \ \forall \ t \le \frac12 (n-1) k ,
\ee
where 
\begin{align}
A(\eta) & = \int_\Sigma \frac{\eta^2}{H} d \sigma ,  
\ \ \ \ B(\eta) = \int_\Sigma \frac{ \eta \dels \eta }{H} d \sigma ,  \label{eq-AB-eta} \\
C(\eta) & = \int_\Sigma \lf[ \frac{ (\dels \eta)^2}{H} - \Pi (\nabs \eta, \nabs \eta) \ri] d \sigma \label{eq-C-eta} . 
\end{align}
Clearly \eqref{eq-equiv-thm}  is equivalent to asserting  that,  for each fixed $\eta$, 
either
\be \label{eq-case-1}
B(\eta) ^2 \le A(\eta) C(\eta) 
\ee
or 
\be \label{eq-case-2}
\frac12 (n-1) k \leq  -  \frac{  B(\eta)}{A(\eta)}  - \sqrt{ \lf(\frac{B(\eta)}{A(\eta)} \ri)^2 - \frac{C(\eta)}{A(\eta) }  }  . 
\ee
This explains how Theorem \ref{thm-poincare-type-ineq-intro} follows from Theorem \ref{thm-poincare-type-ineq}. 

\vh

Next, we apply Theorem \ref{thm-poincare-type-ineq} to eigenvalue estimates on the boundary. 

\begin{cor} \label{cor-eigenvalue-est}
Let $ (\Omega, g)$, $\Sigma$, $k$, $H$, $\Pi$ be given as  in Theorem \ref{thm-poincare-type-ineq}. 
Suppose $\Sigma$ has a component $\Sigma_0$ which is convex, i.e. $\Pi > 0 $ on $\Sigma_0$. 
 Let $ \kappa > 0 $ be  a constant such that 
 $ \Pi \ge \kappa \gamma $, where $\gamma$ is   the induced metric on $\Sigma_0$. 
 Let $ \lam$ be  a  positive eigenvalue of $\dels$ on $(\Sigma_0, \gamma)$.  
If $ \kappa^2 + 2 k > 0 $, then
\be \label{cor-eigen-est}
 \lam \notin \lf( \frac{1}{4} ( n-1) \lf[ \kappa -  \sqrt{ \kappa^2 + 2 k } \ri]^2, \frac{1}{4} ( n-1) \lf[ \kappa +  \sqrt{ \kappa^2 + 2 k } \ri]^2 \ri) . 
 \ee
In particular, if $ k \ge 0 $, the first nonzero eigenvalue $\lam_1(\Sigma_0)$ of $(\Sigma_0, \gamma)$  satisfies  
\be \label{eq-lam-1}
\lam_1(\Sigma_0)  \ge  \frac{1}{4} ( n-1) \lf[ \kappa +  \sqrt{ \kappa^2 + 2 k } \ri]^2   . 
 \ee
\end{cor}

\begin{proof}
By defining $\eta=0$ everywhere on $\Sigma \setminus \Sigma_0$, 
Theorem \ref{thm-poincare-type-ineq} implies 
$$ \int_{\Sigma_0}   \Pi (\nabs  \eta , \nabs  \eta ) d \sigma  
\le    \int_{ \Sigma_0 } \frac{1}{H} \lf(  \dels \eta   +   t  \eta  \ri)^2 d \sigma  , 
 $$
for any $\eta $ defined on $\Sigma_0$ and any  $ t \le \frac{1}{2}(n-1) k$.  
 Let $ A (\eta)$, $B(\eta)$ and $C (\eta) $ be given in \eqref{eq-AB-eta} and \eqref{eq-C-eta} with 
 $\Sigma$ replaced by $\Sigma_0$. 
Suppose $\eta$ is a nonzero eigenfunction, i.e.   $ \dels \eta + \lam \eta = 0 $. Then 
\bee
\begin{split}
B(\eta)^2 - A(\eta) C(\eta) = \lf(  \int_{\Sigmao} \frac{ \eta^2}{H} d \sigma \ri)
\lf( \int_{\Sigmao} \Pi (\nabs \eta, \nabs \eta)  d \sigma \ri)  > 0 .
\end{split}
\eee
Therefore,  \eqref{eq-case-2}  holds,   
which  shows  
\be \label{eq-condition-2-a}
\frac{1}{2} (n-1) k \le \lam - \lf( \frac{  \int_{\Sigmao} \Pi (\nabs \eta, \nabs \eta)  d \sigma}{ \int_{\Sigmao} \frac{ \eta^2}{H} d \sigma } \ri)^\frac12  . 
\ee
On the other hand, 
\be
\int_{\Sigmao} \Pi (\nabs \eta, \nabs \eta)  d \sigma \ge \kappa \int_{\Sigmao} | \nabs \eta |^2 d \sigma = \kappa \lam \int_{\Sigmao} \eta^2 d \sigma 
\ee
and
\be \label{eq-est-H}
\int_{\Sigmao} \frac{\eta^2}{H} d \sigma \le \frac{1}{(n-1) \kappa}  \int_{\Sigmao} \eta^2  d \sigma . 
\ee
Hence, \eqref{eq-condition-2-a} -- \eqref{eq-est-H}  imply  
\be \label{eq-condition-2-c} 
\frac{1}{2} (n-1) k \le \lam- \kappa \sqrt{(n-1) \lam} .
\ee
When  $ \kappa^2 + 2 k > 0 $, it follows from \eqref{eq-condition-2-c} that 
$$ \sqrt{\lam} \notin \lf( \frac{1}{2} \sqrt{ n-1} \lf[ \kappa -  \sqrt{ \kappa^2 + 2 k } \ri] ,
\frac{1}{2} \sqrt{ n-1} \lf[ \kappa +  \sqrt{ \kappa^2 + 2 k } \ri]  \ri) , $$
which completes the proof. 
\end{proof} 

\begin{remark}
Corollary \ref{cor-eigenvalue-est} is motivated by  results in \cite{ChoiWang83, Xia97}. 
If $k = 0 $, \eqref{eq-lam-1} reduces to $\lam_1(\Sigma_0) \ge  (n-1) \kappa^2 $ which 
is the estimate in  \cite{Xia97}. 
\end{remark}

\section{Application to total mean curvature}  \label{sec-apps-bdry-immersion}

In this section, we  recall  the statement of Theorem \ref{thm-immersion-0-intro} -- \ref{thm-immersion-n0-intro}
and give  their proof.
 We begin with the case of nonnegative Ricci curvature. 

\begin{thm} \label{thm-immersion-0}
Let $(\Omega, g)$ be an $n$-dimensional, compact  Riemannian manifold with nonnegative Ricci curvature,
   with connected  boundary $\Sigma$ which has nonnegative mean curvature $H$.
Let $X: \Sigma  \rightarrow \R^m $ be an isometric immersion of $\Sigma$ 
into some  Euclidean space $\R^m$   of dimension  $m\ge n$.  Then 
\begin{equation} 
\label{eq-immersion-0-ineq} 
\int_{\Sigma} H \ d \sigma \le \int_{\Sigma'}  \frac{ | \vec{H}_0 |^2 }{H}   \ d {\sigma} ,
\end{equation}
where  $\vec{H}_0$ is  the mean curvature vector of the  immersion $X$
and $ \Sigma'  \subset \Sigma $ is the set $ \{  \vec{H}_0 (x) \neq 0 \}$.
 Moreover, if equality in (\ref{eq-immersion-0-ineq}) 
holds, then 
\begin{itemize}
\item[a)]  $ H = | \vec{H}_0 | $ identically  on $\Sigma$.
\item[b)]   $(\Omega, g)$ is flat and  $X(\Sigma)$ lies in an $n$-dimensional  plane in  $\R^m$.
\item[c)]   $(\Omega, g)$ is isometric to a domain in $\R^n$ if $X$ is an embedding. 
\end{itemize} 
\end{thm}

\begin{proof}
Since $X$ is an isometric immersion,  one has
\be \label{eq-immersion-0-H-vector}
\dels X = \vec{H}_0 .
\ee
At any $x \in \Sigma$, let  $\{ v_\alpha \ | \ \alpha = 1 , \dots, n -1 \} \subset T_x \Sigma $ be an orthonormal frame 
that  diagonalizes $\Pi $, i.e.
$ \Pi (v_\alpha, v_\beta) = \delta_{\alpha \beta} \kappa_{\alpha} $
where $\{ \kappa_1, \dots, \kappa_{n-1} \}$ are the principal curvature of $\Sigma$ in $(\Omega, g)$ at $x$. 
Let $\{ e_1, \ldots, e_m \}$ denote the standard  basis in $ \R^m$ and 
 $x_i $ be the $i$-th component of $X$.
Then
\be \label{eq-immersion-0-sum-PI}
\begin{split}
\sum_{i=1}^m \Pi (\nabs x_i, \nabs x_i ) = & \ \sum_{i=1}^m \sum_{\alpha, \beta  = 1}^{n-1} 
\Pi (v_\alpha , v_\beta   ) \la e_i , v_\alpha \ra \la e_i, v_\beta \ra \\
= & \ \sum_{\alpha = 1}^{n-1} \kappa_\alpha  = H  .
\end{split}
\ee

Set $k=0$  in Theorem \ref{thm-poincare-type-ineq} and choose $\eta = x_i$,  
 $t = 0 $ in \eqref{eq-thm-bdry-ineq}, we have 
\be \label{eq-immersion-0-bdry-ineq}
 \int_\Sigma  \Pi (\nabs  x_i , \nabs  x_i ) d \sigma  
\le    \int_{ \Sigma }
\frac{1}{H} \lf(  \dels x_i \ri)^2  1_{\Sigma_i'} d \sigma  
\ee
where $ 1_{\Sigma_i'}$ is the characteristic function of the set 
$ \Sigma_i' = \Sigma \setminus  \lf\{  \dels x_i   = 0 \ri\}  $. 
Summing $ \eqref{eq-immersion-0-bdry-ineq}$ over $i$, using \eqref{eq-immersion-0-H-vector}, \eqref{eq-immersion-0-sum-PI}
 and the fact $ \Sigma_i' \subset \Sigma'$,  we have 
\be \label{eq-immersion-0-ineq-pf}
 \int_\Sigma H d \sigma  \le  \int_{\Sigma'}  \frac{1}{H}  | \vec{H}_0 |^2  d \sigma , 
\ee
which proves \eqref{eq-immersion-0-ineq}. 

Next suppose 
\be \label{eq-immersion-0-equality-pf}
 \int_\Sigma H d \sigma =  \int_{\Sigma'}  \frac{1}{H}  | \vec{H}_0 |^2  d \sigma .
\ee
Then it follows from \eqref{eq-immersion-0-bdry-ineq} that 
\be
 \int_\Sigma  \Pi (\nabs  x_i , \nabs  x_i ) d \sigma  
=    \int_{ \Sigma }
\frac{1}{H} \lf(  \dels x_i \ri)^2  1_{\Sigma_i'} d \sigma, \ \ \forall \ i . 
\ee
By the rigidity part of Theorem \ref{thm-poincare-type-ineq}, 
there exist functions $u_i $, $1 \le i \le m  $, such that
 $  u_i = x_i  $ {at}  $ \Sigma$ and 
\be \label{eq-hessian-ui-zero}
 \nabla^2 u_i = 0  \ \mathrm{on} \ \Omega .
 \ee
Moreover, by  \eqref{eq-bdry-a} or \eqref{eq-bdry-b}, we have 
\be \label{eq-immersion-H-relation}
\vec{H}_0 + H  d \Phi  (\nu) = 0  \ \mathrm{at} \ \Sigma ,
\ee
where  $ \Phi : \Omega \rightarrow \R^m $  is a (harmonic)  map defined by 
$ \Phi = (u_1, \ldots, u_m ) $, 
$ d \Phi = ( d u_1, \ldots, d u_m ) $
is the associated  tangent map, and $\nu$ is the unit outward normal to $\Sigma$
in $(\Omega, g)$. 

We claim 
\be \label{eq-nu-x}
d \Phi  (\nu) (x) \neq 0,  \  \ \forall \ x \in  \Sigma.
\ee
To see this,  first consider a point $ y  \in  \Sigma'$ ($\Sigma'  \neq \emptyset$ by \eqref{eq-immersion-0-H-vector}). 
At  $y$, \eqref{eq-immersion-H-relation}  implies 
\be \label{eq-nu-p}
d \Phi  (\nu) (y)  \neq 0 \ \ \mathrm{and} \ \ d \Phi  (\nu)  (y) \perp  X(\Sigma) . 
\ee
Hence, the rank of $ d \Phi $ at $y$ is $n$  
by  \eqref{eq-nu-p} and the fact  $ \Phi |_\Sigma = X $ .
  On the other  hand,  \eqref{eq-hessian-ui-zero} shows $ d u_i $ is parallel on $\Omega$,  $ \forall \ i$.
Therefore, the rank of $ d \Phi $ equals $n$  everywhere on $\Omega$.
In particular, this proves \eqref{eq-nu-x}. 

By  \eqref{eq-immersion-H-relation} and \eqref{eq-nu-x}, we now  have 
\be \label{eq-H-zero-set}
 \{ x \in \Sigma \ | \ H (x) \neq 0 \} = \Sigma' . 
 \ee
Thus  \eqref{eq-immersion-0-equality-pf} becomes  
\be \label{eq-equality-pf-2}
\int_{\Sigma'} H d \sigma = \int_{\Sigma'} H | d \Phi (\nu) |^2 d \sigma 
\ee
by  \eqref{eq-immersion-H-relation}. 
As  $ \Sigma' $ is a nonempty open set in $\Sigma$,  \eqref{eq-H-zero-set} and \eqref{eq-equality-pf-2}  imply 
\be \label{eq-dPhi}
  | d \Phi  (\nu) | (z) = 1 \ \mathrm{and}  \ d \Phi(\nu) (z) \perp X(\Sigma), 
   \  \forall \  z \in \Sigma'  .
  \ee
  It follows from  \eqref{eq-dPhi},  \eqref{eq-immersion-H-relation} and \eqref{eq-H-zero-set}
  that  $H = | \vec{H}_0 | $ identically on $ \Sigma $.

The rest of the claims now follows  from \cite[Proposition 2]{HangWang07}. 
 For completeness, we include  the proof.  
By \eqref{eq-dPhi} and the fact $ \Phi |_\Sigma = X $, one knows  
$ g =   \sum_{i=1}^m d u_i \otimes du_i $ at $\Sigma'$. 
As a result, 
$ g =   \sum_{i=1}^m d u_i \otimes du_i $
on $\Omega$  as  both tensors are parallel.  
Clearly  this shows    $(\Omega, g)$ is flat  and  $\Phi$ is an isometric immersion.
Next, let $v, w$ be any  tangent vectors to $\Omega$.
 \eqref{eq-hessian-ui-zero} implies 
\be \label{eq-totally-geodesic}
0 =  v w (\Phi) - \nabla_v w (\Phi) = \overline{\nabla}_{d \Phi (v) }  ( d \Phi (w) )  - d \Phi (\nabla_v w ) ,
\ee
where $\nabla$ and $\overline{\nabla}$ denote 
the connection on $(\Omega, g)$ and $\R^m$ respectively. 
By definition,  \eqref{eq-totally-geodesic} shows   $\Phi: \Omega \rightarrow \R^m$ is totally geodesic. 
Therefore, $\Phi(\Omega)$ (hence  $X(\Sigma))$ lies in an $n$-dimensional plane in  $\R^m$. 
Without losing generality, one can assume $ \Phi (\Omega) \subset \R^n$. 
If $X: \Sigma \rightarrow \R^m$ is indeed an  embedding, then $X(\Sigma) = \p W $ where $W$ is the closure 
of a bounded domain in $ \R^n$. 
Since  $ \Phi$ is an immersion and $ \Phi |_\Sigma = X $, 
 one can show   $ \Phi (\Omega) \subset  W $ 
 and $\Phi: \Omega \setminus \Sigma \rightarrow W \setminus \p W $
 (by checking that $ \Phi (\Omega) \setminus {W}$ is both open and closed
 in  $ \R^n \setminus W $). 
 On the other hand,   $ \Phi$ being a local isometry 
 implies $\Phi: \Omega \rightarrow W $ is  a covering map.  Therefore,  
 $\Phi$ is a homeomorphism, and hence  an isometry between $\Omega$ and $ W$.
\end{proof}


\begin{thm} \label{thm-immersion-p0}
Let $(\Omega, g)$ be an $n$-dimensional, compact  Riemannian manifold 
with positive Ricci curvature,
 with connected  boundary $\Sigma$ which has nonnegative mean curvature $H$.
Let  $k>0$ be a constant such that  
$$ \Ric \ge (n-1) k g . $$
 Suppose  there exists an isometric immersion $X: \Sigma  \rightarrow \mS^{m}_{k}  $,
 where   $\mS^m_k$ is the  sphere of  dimension $m \ge n$ with constant sectional curvature $k$.  Then 
\begin{equation} 
\label{eq-immersion-p0-ineq} 
\int_{\Sigma} H \ d \sigma < \int_{\Sigma}  \frac{ | \vec{H}_{\mS} |^2  + \frac14  (n-1)^2 k   }{H}   \ d {\sigma} ,
\end{equation}
where  $\vec{H}_{\mS} $ is  the mean curvature vector of the immersion $X$ into $\mS^m_{k}$.
 \end{thm}

\begin{proof}
We identify $\mS^m_k$ with the sphere of radius $\frac{1}{\sqrt{k}}$ centered at the origin in $ \R^{m+1}$, i.e.
$ \mS^m_k = \lf\{ (y_1, \ldots, y_{m+1} ) \subset \R^{m+1} \ | \ \sum_{i=1}^{m+1} y_i^2 = \frac{1}{k}  \ri\} $
and view 
\bee
 X = (x_1, \ldots, x_{m+1}) : \Sigma \longrightarrow \mS^m_k \subset \R^{m+1}
 \eee
as  an isometric immersion of $\Sigma$ into $ \R^{m+1}$. 
Let $ \vec{H}_0$ denote  the mean curvature vector of $X: \Sigma \rightarrow \R^{m+1}$,
then
\be \label{eq-immersion-H-p0}
\vec{H}_0 = \vec{H}_{\mS} +  {k} \la \vec{H}_0 , X \ra X ,   
\ee
where $\la \cdot, \cdot \ra $ is the inner product in $\R^{m+1}$.
Apply the fact 
\be \label{eq-immersion-H-p0-xk}
\Delta_\Sigma X = \vec{H}_0 \ \
\mathrm{and} \ \
\la X, X \ra = \frac{1}{k} ,
\ee
we have 
\be \label{eq-immersion-H-p0-hx}
\begin{split}
0 = & \ \sum_{i=1}^{m+1} \lf(  x_i \Delta_\Sigma x_i  + | \nabla_\Sigma x_i |^2  \ri) \\
= & \ \la \vec{H}_0 , X \ra + (n-1) . 
\end{split}
\ee

In Theorem \ref{thm-poincare-type-ineq}, 
choose $\eta = x_i$ and $ t = \frac12 (n-1) k > 0$ in \eqref{eq-thm-bdry-ineq}, we have 
\be \label{eq-immersion-p0-bdry-ineq}
\begin{split}
 \int_\Sigma  \Pi (\nabs  x_i , \nabs  x_i ) d \sigma  
  <  & \   \int_{ \Sigma'}  \frac{1}{H}  \lf[  \dels x_i   +   \frac12 (n-1) k  x_i   \ri]^2  1_{\Sigma_i'} d \sigma   \\
\end{split}
\ee
where $ 1_{\Sigma_i'}$ is the characteristic function of the set 
$$ \Sigma_i' = \Sigma \setminus  \lf\{  \dels x_i + \frac12 (n-1) k   x_i   = 0 \ri\}  . $$ 
Summing $ \eqref{eq-immersion-p0-bdry-ineq}$ over $i$ and  
using \eqref{eq-immersion-H-p0} -- \eqref{eq-immersion-H-p0-hx},  we have 
\bee \label{eq-immersion-p0-ineq-pf}
\begin{split}
 \int_\Sigma H d \sigma  <  & \  \int_{\Sigma}  \frac{1}{H}  \lf[ | \vec{H}_0 |^2 + (n-1) k  \la \vec{H}_0, X \ra + \frac14 (n-1)^2 k^2  | X|^2 \ri]  d \sigma \\
 = & \  \int_{\Sigma}  \frac{1}{H}  \lf[ | \vec{H}_{\mS} |^2 + \frac14  (n-1)^2 k   \ri]  d \sigma ,
\end{split} 
\eee
where we have also used  \eqref{eq-immersion-0-sum-PI}. 
This  proves \eqref{eq-immersion-p0-ineq}. 
\end{proof}


\begin{thm} \label{thm-immersion-n0}
Let $(\Omega, g)$ be an $n$-dimensional, compact  Riemannian manifold 
 with  boundary $\Sigma$ which has  nonnegative mean curvature $H$.
 Let   $k>0$ be  a constant satisfying 
$$ \Ric \ge - (n-1) k g . $$
 Suppose $\Sigma$ has a component $\Sigma_0$  which admits  an isometric immersion 
 $$X = (t, x_1, \ldots, x_n) : \Sigma_0   \longrightarrow \H^{m}_{-k}  \subset \R^{m,1}, $$
 where   $ \R^{m,1}$ is the $(m+1)$-dimensional Minkowski spacetime  with $m \ge n$ and  
 $$ \H^m_{-k}  = \lf\{ (y_0, y_1, \ldots, y_m  ) \subset \R^{m, 1}  \ | \ - y_0^2 + \sum_{i=1}^{m} y_i^2 = - \frac{1}{k}, \ y_0 > 0  \ri\} .$$ 
Then 
\be  \label{eq-immersion-n0-ineq}
\begin{split}
& \ 
  \int_{\Sigma_0} H  d \sigma +   \int_{\Sigma_0} \Pi (\nabs t , \nabs t) d \sigma  \\
 <     & \  \int_{\Sigma'_0}  \frac{1}{H}   \lf\{ 
    | \vec{H}_{\H} |^2  -  \frac14 (n-1)^2 k    +  \lf[  \dels t   -  \frac12 (n-1) k  t   \ri]^2 \ri\} 
 d \sigma  ,
 \end{split} 
\ee 
where    
$\vec{H}_{\H} $ is  the mean curvature vector of the immersion $X$ into  $\H^m_{-k}$,
\be \label{eq-p-of-big-integrand}
  |\vec{H}_{\H} |^2  -  \frac14 (n-1)^2 k     +  \lf[ \dels t   -  \frac12 (n-1) k  t  \ri]^2  \ge 0  \ \ \mathrm{on} \ \Sigma_0,
  \ee
and  $\Sigma'_0 $ is the set consisting of all $x \in \Sigma_0 $ such that 
$$  
|\vec{H}_{\H} |^2 (x) -  \frac14 (n-1)^2 k   
  +  \lf[ \dels t   -  \frac12 (n-1) k  t   \ri]^2 (x) > 0  . $$ 
 \end{thm}

\begin{proof}
Let $ \vecHm $ be the mean curvature vector of $X: \Sigma_0 \rightarrow  \R^{m,1} $, then
\be \label{eq-immersion-H-n0}
\vecHm = \vec{H}_{\H} - {k} \la \vecHm , X \ra X , 
\ee
and
\be \label{eq-immersion-H-n0-xk}
\Delta_\Sigma X = \vecHm, \ \ 
\la X, X \ra = - \frac{1}{k} ,
\ee
where $\la \cdot, \cdot \ra = - d y_0^2 + \sum_{i=1}^m d y_i^2 $ is the  Lorentzian  product  on $\R^{m,1}$.  
By \eqref{eq-immersion-H-n0-xk}, we have 
\be \label{eq-immersion-H-n0-hx}
\begin{split}
0 = & \  - ( t  \Delta_\Sigma t + | \nabs t  |^2 ) + \sum_{i=1}^{m} \lf(  x_i \Delta_\Sigma x_i  + | \nabla_\Sigma x_i |^2  \ri) \\
= & \ \la \vecHm , X \ra + (n-1) . 
\end{split}
\ee

At any $x \in \Sigma_0$, let  $\{ v_\alpha \ | \ \alpha = 1 , \dots, n -1 \}  $ be an orthonormal frame in $T_x \Sigma_0$ 
such that $ \Pi (v_\alpha, v_\beta) = \delta_{\alpha \beta} \kappa_{\alpha} $
where $\{ \kappa_1, \dots, \kappa_{n-1} \}$ are the principal curvature of $\Sigma_0$ in $(\Omega, g)$ at $x$. 
We have
\be
 \Pi( \nabs t, \nabs t  )  =      \sum_{\alpha =1}^{n-1} \kappa_\alpha \la \p_{y_0}, v_\alpha \ra^2 
\ee
and
\be 
\begin{split}
  \sum_{i=1}^{m} \Pi( \nabs x_i, \nabs x_i )  
 = & \   \sum_{\alpha =1}^{n-1}  \kappa_{\alpha}  \lf( \sum_{i=1}^m \la \p_{y_i} , v_\alpha \ra^2 \ri) \\
= & \   \sum_{\alpha =1}^{n-1}  \kappa_{\alpha}  \lf( 1 +  \la  \p_{y_0} , v_\alpha \ra^2 \ri) .
\end{split} 
\ee
Therefore,
\be \label{eq-H-sum-PIt}
 \int_{\Sigma_0} H  d \sigma +   \int_{\Sigma_0} \Pi (\nabs t , \nabs t) d \sigma
 =   \sum_{i=1}^{m} \int_{\Sigma_0} \Pi( \nabs x_i, \nabs x_i )  d \sigma .
\ee

Now 
 choose $\eta = x_i$ on $\Sigma_0$,  $\eta = 0 $ on $\Sigma \setminus \Sigma_0$, 
and   $ t = - \frac12 (n-1) k <  0$  in Theorem \ref{thm-poincare-type-ineq}, we have 
\be \label{eq-immersion-n0-bdry-ineq}
\begin{split}
 \int_{\Sigmao}  \Pi (\nabs  x_i , \nabs  x_i ) d \sigma  
  <  & \   \int_{ {\Sigma}'_{0i} }  \frac{1}{H}  \lf[  \dels x_i   -  \frac12 (n-1) k  x_i   \ri]^2  1_{\Sigma_{0i}'} d \sigma   \\
\end{split}
\ee
where $ 1_{\Sigma_{0i}'}$ is the characteristic function of the set 
$$ {\Sigma}'_{0i} = \Sigmao \setminus  \lf\{  \dels x_i - \frac12 (n-1) k   x_i   = 0 \ri\}  . $$ 
Direct calculation   using \eqref{eq-immersion-H-n0} -- \eqref{eq-immersion-H-n0-hx} shows
\be \label{eq-pointwise-sum}
\begin{split}
& \  \sum_{i=1}^m \lf[  \dels x_i   -  \frac12 (n-1) k  x_i   \ri]^2   \\
= & \  | \vec{H}_{\H} |^2  -  \frac14 (n-1)^2 k    +  \lf[  \dels t   -  \frac12 (n-1) k  t \ri]^2  ,
\end{split}
\ee
which also proves \eqref{eq-p-of-big-integrand}.
Summing $ \eqref{eq-immersion-n0-bdry-ineq}$ over $i=1, \ldots, m $ and using \eqref{eq-H-sum-PIt} -- \eqref{eq-pointwise-sum} 
together  with the fact ${\Sigma}'_{0i}  \subset {\Sigma}'_0$,  we have 
\bee \label{eq-immersion-n0-ineq-pf}
\begin{split}
& \ 
  \int_{\Sigma_0} H  d \sigma +   \int_{\Sigma_0} \Pi (\nabs t , \nabs t) d \sigma  \\
 <     & \  \int_{\Sigma'_0}  \frac{1}{H}   \lf\{ 
    | \vec{H}_{\H} |^2  -  \frac14 (n-1)^2 k     +  \lf[  \dels t   -  \frac12 (n-1) k  t   \ri]^2 \ri\} 
 d \sigma   .
 \end{split} 
\eee
This completes the proof.
\end{proof}


\section{Rigidity results} \label{rigidity}

 Inequalities  \eqref{eq-immersion-p0-ineq} and \eqref{eq-immersion-n0-ineq} 
 are not sharp in the context of Theorem \ref{thm-immersion-p0} and Theorem \ref{thm-immersion-n0}. 
In these cases,  one wonders if there exist  sharp integral  inequalities 
involving $H$ and $ | \vec{H}_{\mS} |$ (or $ | \vec{H}_{\H} | $)
which include a rigidity statement  in the case of equality. 

In what follows,  by scaling the metric,  we assume  
 $\Ric \geq (n-1) g$ or  $\Ric\geq -(n-1 ) g $. 
In the  latter case,  the scalar curvature $R $ of $g$ satisfies  $ R  \ge - n (n-1)$.
By the results in  \cite{WangYau07, ShiTam07, Kwong}, 
 there exists  a sharp integral  inequality  relating  $H$ and $ | \vec{H}_{\H} |$   if 
 the manifold $\Omega$ is  spin and the
boundary $\Sigma$ embeds isometrically in the hyperbolic space $\mathbb{H}^n$ as a convex hypersurface.
On the other hand, 
the counterexample to Min-Oo's conjecture 
in \cite{BMN11} shows   that even the  pointwise condition   $ H = | \vec{H}_{\mS} |$  
is not sufficient to guarantee  rigidity  if  one only assumes  $R \ge n(n-1)$.
This gives rise to the  following  rigidity  question: 

\begin{ques} \label{ques-rigidity-H}
Let $\left( \Omega ,g\right)  $  be an $n$-dimensional,  compact Riemannian manifold
with boundary $\Sigma $.  Let $ D \subset \mS^n_+ $ be a bounded domain  with smooth boundary $\p D$, 
where $ \mS^n_+$ is   the standard  $n$-dimensional hemisphere.
Suppose
\begin{itemize}
\item $\Ric\geq\left(  n-1\right)  g$

\item there exists an isometry $X : \Sigma \rightarrow \p D $ 

\item  $ H \ge H_{\mS} \circ X  $, where $H$, $H_{\mS}$ are the mean curvature of $ \Sigma$, $\p D$ in 
$(\Omega, g)$, $\mS^n_+$ respectively. 
\end{itemize}

 \noindent Is $\left(  \Omega ,g\right)  $  isometric to $D$ in  $\mathbb{S}_{+}^{n}$?
\end{ques}
 
At this stage,  we do not know the  answer to Question \ref{ques-rigidity-H}. 
However,  it was shown  in \cite{HangWang09} that  Question \ref{ques-rigidity-H} 
has an affirmative answer if  the assumption $ H \ge H_{\mS} \circ X  $ is replaced by a stronger 
assumption on  the second fundamental forms.

\begin{thm} [\cite{HangWang09}] \label{thm-HangWang}
Let $\left( \Omega ,g\right)  $  be an $n$-dimensional,  compact Riemannian manifold
with boundary $\Sigma $.  Let $ D \subset \mS^n_+ $ be a bounded domain  with smooth boundary $\p D$, 
where $ \mS^n_+$ is   the standard  $n$-dimensional hemisphere.
Suppose
\begin{itemize}
\item $\Ric\geq\left(  n-1\right)  g$

\item there exists an isometry $X : \Sigma \rightarrow \p D $ 

\item  $ \Pi  \ge \Pi_{\mS} \circ X  $, where $\Pi$, $\Pi_{\mS}$ are the  second fundamental form of $ \Sigma$, $\p D$ in 
$(\Omega, g)$, $\mS^n_+$ respectively. 
\end{itemize}
Then $\left(  \Omega , g\right)  $ is isometric to $D$  in  $\mathbb{S}_{+}^{n}$. 
\end{thm}

In the rest of this section,   we prove  Theorem \ref{thm-spherical-intro} and \ref{thm-hyperbolic-intro}, which are 
analogues of Theorem \ref{thm-HangWang} when  the boundary is  only  isometrically 
immersed in  a sphere or in a hyperbolic space of higher dimension.

\begin{thm} \label{thm-spherical} 
Let $\left( \Omega ,g\right)  $  be an $n$-dimensional,  compact Riemannian manifold
with boundary $\Sigma $.  Suppose 
\begin{itemize}
\item $\Ric\geq\left(  n-1\right)  g$

\item there exists  an isometric immersion $X  :  \Sigma \rightarrow\mathbb{S}^{m}$, where $ \mS^m$ is a standard  sphere of 
dimension $m \ge n $

\item $\Pi\left(  v,v \right)  \geq\left\vert \Pi_\mS \left( v ,  v \right)  \right\vert $, for any $ v \in T\Sigma$.
Here $\Pi$ is the second fundamental form of $\Sigma$  in $(\Omega, g)$ and 
 $\Pi_\mS$ is the vector-valued, second fundamental form of  the immersion $X $. 
\end{itemize}
Then $\left(  \Omega ,g\right)  $ is spherical, i.e. having constant sectional curvature $1$.  
\end{thm}

We divide the proof of  Theorem \ref{thm-spherical} into a few steps. First,  we fix some notations. 
Let $ \overline{\nabla} $  denote the covariant derivative on $\mS^m$, which is identified with 
 the  unit sphere centered at the origin  in $ \R^{m+1}$. 
 Given any  $\alpha = (\alpha_1, \ldots, \alpha_{m+1})  \in \mathbb{S}^{m}$, let  
 $ F = F_\alpha $ be the restriction of the linear function
 $  \la \alpha, x \ra = \alpha_1 x_{1}+\cdots+\alpha_{m+1}x_{m+1}$ to $ \mS^m$. 
 The gradient of $F$ on $ \mS^m$, denoted by $ \overline{\nabla} F$,   is  
\be \label{eq-grad-F}
\overline{\nabla}F\left(  x\right)  =\alpha-\left\langle \alpha,x\right\rangle
x, \ x \in \mS^m.
\ee
On $\Sigma$,  define  $f  =  F \circ X  $.
For simplicity,  given  any  $ p \in \Sigma$,  we  let $\overline{\nabla}^{\perp}F $ be  the component
of $\overline{\nabla}F  ( X  (p) )  $ 
normal to $X _* \left(  T_{p}\Sigma\right)  $. 
Given $v,w \in T_{p}\Sigma$, recall that 
 $\Pi_\mS \left( v , w\right)  =\left(  \overline{\nabla}_{X _* (v)}   X _* (w) 
\right)  ^{\bot}$ 
is the component of $  \overline{\nabla}_{X _* (v)}   X _* (w)  $ normal to $X _{\ast}\left(  T_{p}\Sigma\right)  $.
We let   $ \vec{H}_{\mS}$ denote  the mean curvature vector of  $X $, which is the trace of $\Pi_\mS$.

\begin{lemma} \label{lem-identities}  
Along $\Sigma$, one has
\begin{align}
f^{2}+\left\vert \nabs f\right\vert ^{2}+\left\vert \overline{\nabla
}^{\perp}F \right\vert ^{2}  &  =1,  \label{eq-id-1}\\
\dels f + \left(  n-1\right)  f-\left\langle \vec{H}_{\mS}  ,\overline{\nabla}
^{\perp}F  \right\rangle  &  = 0 , \label{eq-id-2} \\
\left\langle \overline{ \nabla}_{X _* ( \nabs f) } \overline{\nabla}^{\perp}F  ,  \vec{n} \right \rangle
 +\left\langle \Pi_{\mS} \left(  \nabs f ,  \nabs f \right)  , \vec{n}  \right\rangle  &  =0. \label{eq-id-3}
\end{align}
Here $ \vec{n}$ is any vector that is normal to $ X  (\Sigma) $ in $\mS^m$. 

\end{lemma}

\begin{proof}
\eqref{eq-id-1} follows from the fact $F^{2}+\left\vert \overline{\nabla
}F\right\vert ^{2}=1$. 
To show \eqref{eq-id-2} and \eqref{eq-id-3}, we note that  $F$ on $\mS^m$ satisfies 
\be  \label{eq-hessian-F} 
\overline{\nabla}^{2}F=-Fg_{_\mS} ,
\ee
where $g_{_\mS}$ is the standard metric on $\mS^m$. 
 \eqref{eq-hessian-F} readily implies
 \be \label{eq-hessian-f}
 \nabs^2 f (v, w) - \la  \Pis( v, w ) ,  \onab^\perp F \ra = - f \la v, w \ra  , \ \ \forall  \ v, w \in T\Sigma,
 \ee
where $\nabs^2 $ denotes the Hessian on $\Sigma$. Taking trace of \eqref{eq-hessian-f} gives \eqref{eq-id-2}.
\eqref{eq-hessian-F} also implies , 
\begin{align*}
0 &  = \overline{\nabla}^{2}F\left( X _* (  \nabs f )  , \vec{n} \right)  \\
&  = \left\langle \overline{ \nabla}_{X _* ( \nabs f) } \overline{\nabla}^{\perp}F  , \vec{n}  \right\rangle 
+\left\langle \Pi_{\mS} \left(  \nabs f ,  \nabs f \right)  , \vec{n}  \right\rangle ,
\end{align*} 
which proves \eqref{eq-id-3}.
\end{proof}

The condition   $ \Pi (v, v)  \ge | \Pi_{\mS} (v, v)  |$, $ \forall \ v \in T\Sigma$, 
implies $H \ge  | \mHs| \ge 0 $. 
 By Reilly's theorem  (\cite[Theorem 4]{Reilly77}),   to prove Theorem \ref{thm-spherical}, it 
suffices to assume  $\lambda_{1}  >n$, where $\lambda_1$ is the first Dirichlet eigenvalue of  $(\Omega, g)$. 
Under this assumption, we  let $u  $ be the unique  solution to %
\be \label{eq-def-u}
\left\{
\begin{array}{rcl}
 \Delta u + n u & = &   0  \ \  \text{on} \  \Omega,\\
u & = & f  \  \ \text{at} \ \Sigma.
\end{array}
\right.
\ee
On $(\Omega, g)$, define 
$$ \phi=\left\vert \nabla u\right\vert ^{2}+u^{2} .$$
 A basic fact about  $\phi$ is that it is subharmonic, which follows from  
\begin{equation} 
\frac{1}{2}\Delta\phi=  \left\vert \nabla^{2}u+ug\right\vert ^{2}+\mathrm{Ric}(\nabla
u,\nabla u)-\left(  n-1\right)  \left\vert \nabla u\right\vert ^{2} \ge 0 . \label{Bo} \\
\end{equation}
 On $\Sigma$,   define $ \displaystyle \chi=\frac{\partial u}{\partial\nu}$,
 where $\nu$ is the unit outward normal to $\Sigma $ in $(\Omega, g)$. 
  Then  
\begin{equation} \label{eq-phi-sigma}
\phi|_{\Sigma}=\left\vert \nabla_{\Sigma}f\right\vert ^{2}+\chi^{2}%
+f^{2}=1+\chi^{2}-\left\vert \overline{\nabla}^{\perp}F\right\vert ^{2} %
\end{equation}
by \eqref{eq-id-1} in Lemma \ref{lem-identities}. 

\begin{lemma} \label{lem-phi-nu}
Along $\Sigma$,  the normal derivative of $\phi$ is given by 
\[
\frac{1}{2}\frac{\partial\phi}{\partial\nu}=\left\langle \nabs f,\nabs \chi\right\rangle -\Pi\left(
\nabs f,\nabs f\right)  -H\chi^{2}  -\left\langle \vec{H}_{\mS} %
,\overline{\nabla}^{\perp}F\right\rangle \chi.
\]
\end{lemma}

\begin{proof}
Direct calculation gives  
\be \label{eq-phi-nu}
\begin{split}
\frac{1}{2}\frac{\partial\phi}{\partial\nu}  &  = \nabla^{2}u\left(  \nabla
u,\nu\right)  +f\chi\\
&  = \nabla^{2}u\left(  \nabs f,\nu\right)  +\chi\left[  \nabla^{2}u\left(
\nu,\nu\right)  +f\right] \\
&  = \left\langle \nabs f,\nabs \chi\right\rangle -\Pi\left(\nabs f,\nabs f\right) +\chi\left[  \nabla^{2}u\left(
\nu,\nu\right)  +f\right] .
\end{split}
\ee
By   \eqref{eq-id-1} and   \eqref{eq-def-u}, at $\Sigma$  we have 
\begin{align*}
-nf=\Delta u  &=\Delta_{\Sigma}f+H\chi+\nabla^{2}u\left(  \nu,\nu\right) \\
 & =-\left(  n-1\right)  f+\left\langle \vec{H}_{\mS},\overline{\nabla
}^{\perp}F\right\rangle +H\chi+\nabla^{2}u\left(  \nu,\nu\right)  ,
\end{align*}
which gives 
\be \label{eq-bdry-lap}
\nabla^{2}u\left(  \nu,\nu\right)  +f=-\left\langle \vec{H}_{\mS} ,\overline{\nabla}^{\perp}F\right\rangle -H\chi . 
\ee
The lemma follows from \eqref{eq-phi-nu} and \eqref{eq-bdry-lap}. 
\end{proof}

\vspace{0.1cm}

\begin{proof}[Proof of Theorem \ref{thm-spherical}]
Given any  $q \in \Sigma$, 
we choose   $\alpha = X  (q) \in \mS^m $. 
Then
  $ \overline{\nabla} F ( X (q) ) = 0 $ by \eqref{eq-grad-F}. 
Hence, 
\be \label{eq-phi-q}
 \nabs f (q) = 0 ,  \ \overline{\nabla}^\perp F (q) = 0 \ \ \mathrm{and} \ \  \phi (q) = 1 + \chi^2 (q) . 
\ee
Consider    $ p \in \Sigma$ such that 
$\phi (p) = \max_\Omega \phi  $. By
\eqref{eq-phi-sigma} and \eqref{eq-phi-q},   
\be \label{eq-chi-grad}
\chi^{2}   (p) \geq\left\vert  \onab^{\perp}F\right\vert ^{2} (p) .
\ee
Since $ \la \nabs f , \nabs  \phi \ra (p) = 0 $,  taking $\vec{n} = \onab^\perp F$ in \eqref{eq-id-3}, 
at $p$ we  have 
\be \label{eq-nabs-phi}
\begin{split}
\chi\left\langle \nabs f,\nabs \chi\right\rangle  &
=\left\langle \onab_{ X _* ( \nabs f) }\overline{\nabla}^{\perp}F,\overline{\nabla
}^{\perp}F\right\rangle \\
&  =-\left\langle \Pis \left(  \nabs f,\nabs f\right)
,\onab^{\perp}F\right\rangle  . 
\end{split}
\ee
If $\chi\left(  p\right)  \neq0$, it follows from Lemma \ref{lem-phi-nu}, \eqref{eq-chi-grad} and \eqref{eq-nabs-phi} that 
\be \label{eq-sign-1}
\begin{split}
\frac12 \frac{\partial\phi}{\partial\nu} (p) = & -\frac{1}{\chi}\left\langle \Pis \left(
\nabs f,\nabs f\right)  ,\onab^{\perp}F\right\rangle -\Pi\left(  \nabs f,\nabs f\right) \\
& \   -\left\langle \mHs,\overline{\nabla}^{\perp}F\right\rangle
\chi-H\chi^{2} \\
  \leq & \  \left\vert \Pis \left(  \nabs f,\nabs f\right)
\right\vert -\Pi\left(  \nabs f,\nabs f\right)   +\left(
\left\vert \mHs \right\vert -H\right)  \chi^{2}\\
  \leq & \ 0 .
\end{split}
\ee
If $\chi\left(  p\right)  =0$, then Lemma \ref{lem-phi-nu} gives 
\be \label{eq-chi-0-pnu}
\frac12 \frac{\partial\phi}{\partial\nu} (p) = \la \nabs f , \nabs \chi \ra - \Pi (\nabs f, \nabs f ).
\ee
Moreover, 
 $ \onab^\perp F (p) = 0 $ by \eqref{eq-chi-grad}. 
Taking the second order derivative of $\phi$ along $\nabs f $ at $p$, we have
\be \label{eq-2nd-d}
\begin{split}
0 \ge & \  \frac12 \nabs f ( \nabs f ( \phi ) )   (p) \\
= & \  \nabs f \lf( \chi \la \nabs f , \nabs \chi \ra - \la \onab_{X _*( \nabs f) }
\onab^\perp F, \onab^\perp F \ra \ri) \\
= & \   \la \nabs f , \nabs \chi \ra^2  -  \lf|  \onab_{X _*( \nabs f) }
\onab^\perp F \ri|^2 .
\end{split}
\ee
We claim, at $p$, 
 \be \label{eq-equal-pi}
   ( \onab_{ X _* ( \nabs f) }\overline{\nabla}^{\bot}F )  =- \Pis \left(  \nabs f,\nabs f\right) .
 \ee 
To see this,  take  any $v \in T\Sigma$,  \eqref{eq-hessian-F} and \eqref{eq-phi-sigma} imply 
\be \label{eq-no-t}
\begin{split}
& \ \left\langle \onab_{ X _* ( \nabs f) }\overline{\nabla}^{\bot}F, X _* (v)  \right\rangle \\
= & \ \onab^2 F ( X _* (\nabs f),  X _* (v)   )   - \left\langle \onab_{ X _* ( \nabs f) } X _* (\nabs f) , X _* (v)  \right\rangle \\
= & \ - f v (f) -  \nabs^2 f ( \nabs f, v)  \\
= & \ -\frac{1}{2} v \left(  f^{2}+\left\vert \nabs f\right\vert ^{2}\right) 
  = -\frac{1}{2}v \left(  \phi-\chi^{2}\right)    . 
\end{split}
\ee
Clearly,  $ v \left(  \phi -\chi^{2}\right) $ varnishes at $p$.  Hence,  
 $ ( \onab_{ X _* ( \nabs f) }\overline{\nabla}^{\bot}F ) (p) $ is normal  to $X _* ( T_p \Sigma) $. 
This, together with \eqref{eq-id-3}, implies \eqref{eq-equal-pi}. 
Now it follows from  \eqref{eq-chi-0-pnu}, \eqref{eq-2nd-d} and \eqref{eq-equal-pi} that   
\bee \label{eq-sign-2}
\frac{1}{2}\frac{\partial\phi}{\partial\nu} (p) \leq\left\vert \Pis \left(
\nabla_{\Sigma}f,\nabs f\right)  \right\vert -\Pi\left(
\nabla_{\Sigma}f,\nabs f\right)  \leq0 .
\eee

By  the strong maximum principle (precisely the Hopf boundary point lemma),  we conclude that  $\phi$ must be
a constant.  Hence,  
$  \nabla^{2}u=-ug  $  by (\ref{Bo}).    
Moreover, by \eqref{eq-phi-sigma} and  \eqref{eq-phi-q}, 
\be  \label{eq-constant-c}
\chi^{2}-\left\vert \overline{\nabla}^{\perp}F\right\vert ^{2} = c 
\ee
for some  constant $c \ge 0$. 
We have the following two cases: 

If  $c>0$,   then  $\chi^{2} > \vert \onab^{\perp}F \vert^{2} \ge 0 $. This together with 
 \eqref{eq-sign-1} and the fact $\phi$ is a constant   implies 
$\vert \mHs \vert =H=0 .$
Since $\Pi\geq0$, we have  $\Pi=0$ and   $\Pis = 0 $. 
Thus, $ X  : \Sigma \rightarrow \mS^m$ is totally geodesic. Hence  $X  (\Sigma) $  lies in an $(n-1)$-dimensional 
standard sphere $\mS^{n-1} \subset \mS^m$.  Since $X : \Sigma \rightarrow \mS^{n-1}$ is an isometric immersion,  
we have  $ X  (\Sigma) = \mS^{n-1}$; moreover $X : \Sigma \rightarrow \mS^{n-1}$ is one-to-one
as $\mS^{n-1}$ is simply connected. Therefore, $\Sigma$ is isometric to $\mS^{n-1}$ and is totally geodesic in $(\Omega, g)$.
By \cite[Theorem 2]{HangWang09}, we conclude that $(\Omega, g)$ is isometric to a standard hemisphere $\mS^n_+$. 

If $c = 0 $, then 
  $\phi = 1 $ on $\Sigma$ (and hence on $\Omega$).  
In this case, along  $\Sigma$,  
\[
\left\vert \nabla u\right\vert ^{2}=\left\vert \nabs f\right\vert
^{2}+\chi^{2}=\left\vert \nabs   f\right\vert ^{2}+\left\vert
\onab^{\perp}F\right\vert ^{2}=\left\vert \overline{\nabla }F\right\vert ^{2 } \circ X  .
\]
In particular,  $\nabla u\left(  q \right) =0$ by \eqref{eq-phi-q}. 
We also note that  $u\left(  q \right) = f(q) = 1$ by the definition of $F$. 

Finally, we are in a position  to show $(\Omega, g)$ has constant sectional curvature $1$. 
It suffices to assume  $(\Omega, g)$ is not isometric to $\mS^n_+$.
Given any $x $ in the interior of  $\Omega$, let $q_x \in \Sigma$  such that
$ dist(x, \Sigma) = dist (x, q_x) $, where $ dist(\cdot, \cdot)$ denotes the distance functional on $(\Omega,g )$.
Consider the function $f$ and $u$ constructed in the above proof by taking $q = q_x$. 
Since $(\Omega, g)$ is  not isometric to $\mS^n_+$, 
 the constant $c$ in \eqref{eq-constant-c} must be  $0$, hence   $u$ satisfies 
\be \label{eq-ux}
\nabla^2 u  = - u g , \ \nabla u (q_x) = 0 , \ \ u(q_x) = 1. 
\ee
Let $ \gamma : [0, L] \rightarrow (\Omega, g)$ be the geodesic satisfying  
$ \gamma (0) = q_x$, $ \gamma (L) = x $ and $ L = dist (x, \Sigma)$. 
Let $ \xi = \gamma'(0) $.  Given any constant $l \in  (0,  L)$,  there exists an open neighborhood $W$ of $ \xi$ in $\mS^{n-1}$ 
 such that the exponential map $ exp_{q_x} (\cdot, \cdot) $ is a diffeomorphism from  $(0, l) \times W \subset \R^+ \times \mS^{n-1} $ 
 onto its image in $(\Omega, g)$. 
 Now it is a standard fact that \eqref{eq-ux} implies
 \be
 (\exp_{q_x})^* ( g) = dr^2 + (\sin r)^2 g_{_{\mS^{n-1}}} 
 \ee
 on $ (0, l) \times W $,  where $g_{_{\mS^{n-1}}}$ is the standard metric on $\mS^{n-1}$
(cf.  \cite{HangWang09} for details). Therefore,  $g$  has constant
sectional curvature $1$ at $\gamma (t )$ for any $t < L$.  By continuity,  $g$ has constant sectional 
curvature $1$ at $x $. 
This completes the proof that $\Omega$ is spherical.
\end{proof}


As an application, we have the following rigidity result which is 
a spherical analogue of  \cite[Theorem 1]{HangWang07}. 

\begin{cor} \label{cor-simple-bdry-s}
 Let $\left( \Omega ,g\right)  $  be an $n$-dimensional,  compact Riemannian manifold
with boundary $\Sigma $. Suppose  
\begin{itemize}
\item  $\Ric \ge (n-1) g $ 
\item  $g$ has constant sectional curvature $1$ at every point on $\Sigma$. 
\end{itemize}
If $\Sigma$ is simply connected with nonnegative second fundamental form $\Pi$,
 then $(\Omega, g)$ is isometric to a domain in $\mathbb{S}^n_+$. 
\end{cor} 
\begin{proof}
Let $ R^\Sigma (\cdot, \cdot, \cdot, \cdot)$,  $\nabla^\Sigma$ denote the curvature tensor,  the connection 
 on $\Sigma$ respectively. By  the Gauss equation and the Codazzi equation,  
\bee
\begin{split}
R^{\Sigma}\left(  v_1, v_2 , v_3, v_4 \right)  = & \ \la v_1, v_3 \ra \la v_2, v_4\ra - \la v_1, v_4 \ra \la v_2, v_3 \ra \\
& \ +\Pi\left(  v_1 , v_3\right)  
\Pi\left( v_2 , v_4 \right)  -\Pi\left(  v_1 , v_4 \right)  \Pi\left(  v_2, v_3 \right)  ,\\
 0 = & \ \left(  \nabla^\Sigma_{v_1}\Pi\right)  \left(  v_2, v_3\right)  -\left(  \nabla^\Sigma_{v_2}%
\Pi\right)  \left(  v_1, v_3 \right)   
\end{split} 
\eee
where  $v_1, \ldots v_4 \in T\Sigma$.
As    $\Sigma$ is simply connected,   the fundamental theorem of
hypersurfaces (cf. \cite{Sp}) implies  there exists an isometric immersion $ \Phi : \Sigma \rightarrow \mathbb{S}^{n}$ with $\Pi$ as its
second fundamental form.  Since $\Pi\geq0$, by a result of Do Carmo and Warner in \cite{dCW70}, 
  $ \Phi $ is  an embedding and $\Phi (\Sigma)$ is   a convex
hypersurface in a hemisphere $\mS^n_+$.
Now apply Theorem \ref{thm-spherical}, we conclude  that $(\Omega, g)$ has constant sectional curvature $1$ everywhere. 
Let $D$ be the region bounded by $ \Phi (\Sigma) $ in $\mS^n_+$. 
We glue $\Omega$ and 
$\mS^n \setminus D$ along the boundary via the isometric embedding $\Phi$ to get a closed  manifold  $(\widetilde{{M}}, \tilde{g})$. 
The fact that   $\Sigma$ has the same second fundamental form in $(\Omega, g)$ and $\mS^n$ 
and that  both $(\Omega, g)$ and $\mS^n$ have constant sectional curvature $1$ imply that 
$\tilde{g}$ is $C^2$ (indeed smooth) across $\Sigma$ in $\widetilde{{M}}$. Hence, $(\widetilde{M}, \tilde{g})$ is a spherical space form
whose volume is greater than half of $\mS^n$ (because it contains  a hemisphere). 
Therefore, $(\widetilde{M}, \tilde{g})$ is isometric to $\mS^n$ and  $(\Omega, g)$ is isometric to a domain in $\mS^n_+$. 
\end{proof}

\begin{remark}
One can also apply  Theorem \ref{thm-HangWang} in the above  proof.
\end{remark}

Theorem \ref{thm-spherical-intro} now follows from Theorem \ref{thm-spherical} and Corollary \ref{cor-simple-bdry-s}.
When  $\Ric\geq-\left(  n-1\right)  $, similarly we  have 

\begin{thm} \label{thm-hyperbolic}
Let $\left(  \Omega, g\right)  $ be an $n$-dimensional,  compact Riemannian manifold with
boundary $\Sigma$. Suppose

\begin{itemize}
\item $\Ric\geq-\left(  n-1\right)  g$

\item there is an isometric immersion $X:   \Sigma  \rightarrow\mathbb{H}^{m}$, 
where $\mathbb{H}^{m}$ is a hyperbolic space of dimension $m\geq n$

\item $ \Pi\left(  v,v\right)  \geq \left\vert \Pi_{\mathbb{H}}\left(v,v\right)  \right\vert $, for any $v\in T\Sigma$. 
Here $\Pi$ is the second fundamental form of $\Sigma$  in $(\Omega, g)$ and 
$\Pi_{\mathbb{H}}$ is the vector-valued, second fundamental form of  the immersion $X $. 
\end{itemize}
Then $\left( \Omega,g\right)  $ is hyperbolic, i.e. having constant sectional curvature $-1$. 
\end{thm}

The proof is parallel to that of Theorem \ref{thm-spherical}.  
Let $\la \cdot, \cdot \ra$ denote the dot  product on 
$\mathbb{R}^{m,1}$ and $\onab $ be the connection on $\mathbb{H}^m$.  Identify 
$\mathbb{H}^{m}$ with $\left\{  x\in\mathbb{R}^{m,1} \ | \ \left\langle 
x,x  \right\rangle =-1,x_{0}>0\right\}  $. 
For  any $\alpha \in X(\Sigma) \subset \mathbb{H}^{m}$, consider $F\left(  x\right)  =F_{\alpha}\left(
x\right)  =  \left\langle \alpha,x\right\rangle  $ on
$\mathbb{H}^{m}$. Its gradient is 
$
\overline{\nabla}F\left(  x\right)  =\alpha+\left\langle 
\alpha,x\right\rangle  x.
$
Thus $\left\vert \overline{\nabla}F\left(  x\right)  \right\vert
^{2} =-1+F^{2}$.
Given any $ p \in \Sigma$, let $\overline{\nabla}^{\perp}F\circ X\left(  p\right)  $ be the
component of $\onab F \circ X (p)$ orthogonal to $X_\ast\left(  T_{p}\Sigma\right)  $. On $\Sigma$,  define 
$f=F\circ X$.  Let $u $ be the smooth solution to
\[
\left\{
\begin{array}[c]{ccc}
\Delta u=nu & \text{on} & \Omega\\
u=f & \text{at} & \Sigma
\end{array}
\right.
\]
and define $\chi = \frac{\p u}{\p \nu}$.
Then $\phi:=\left\vert \nabla u\right\vert ^{2}-u^{2}$ is subharmonic as seen
from 
\[
\frac{1}{2}\Delta\phi=\left\vert \nabla^{2}u-ug\right\vert ^{2}+\mathrm{Ric}(\nabla
u,\nabla u)+\left(  n-1\right)  \left\vert \nabla u\right\vert ^{2} .
\] 
Similar to \eqref{eq-phi-sigma}, 
we have%
\begin{align*}
\phi|_{\Sigma}  & =\left\vert \nabs f\right\vert ^{2}+\chi^{2}%
-f^{2}=-1+\chi^{2}-\left\vert \onab^{\perp}F\right\vert ^{2} .
\end{align*}
By analyzing the normal derivative $ \frac{\partial\phi}{\partial\nu}$ in the same way as in the proof of Theorem \ref{thm-spherical}, 
we conclude by the strong maximum
principle that $\phi$ is constant.  Therefore,  $\nabla^{2}u=ug$ and $\chi^{2}-\left\vert \overline{\nabla}^{\perp
}F\right\vert ^{2}$ is a nonnegative constant $c$ along $\Sigma$. If $c>0$, 
it implies  $\Pi_{\mathbb{H}}=0$, which contradicts  the fact  $\mathbb{H}^{m}$ does not contain a
closed totally geodesic submanifold. Therefore $\chi^{2}-\left\vert
\overline{\nabla}^{\perp}F\right\vert ^{2}=0$ at $\Sigma$, which shows   $\phi=-1$ on $\Omega$.
By the same argument as that of Theorem \ref{thm-spherical}, we conclude that $(\Omega , g)$
has constant sectional curvature $-1$. 

\begin{cor} \label{cor-simple-bdry-H}
 Let $\left( \Omega ,g\right)  $  be an $n$-dimensional,  compact Riemannian manifold
with boundary $\Sigma $. Suppose  
\begin{itemize}
\item  $\Ric \ge - (n-1) g $ 
\item  $g$ has constant sectional curvature $-1$ at every point on $\Sigma$. 
\end{itemize}
If $\Sigma$ is simply connected with nonnegative second fundamental form $\Pi$,
 then $(\Omega, g)$ is isometric to a domain in $\mathbb{H}^n$. 
\end{cor}


The proof is similar to that of Corollary \ref{cor-simple-bdry-s}.
Since $g$ has constant sectional curvature $-1$ along $\Sigma$ and 
 $\Sigma$ is simply connected,  by the Gauss and Codazzi equations and the fundamental theorem of
hypersurfaces (cf. \cite{Sp}), 
there exists an isometric immersion $ \Phi : \Sigma \rightarrow \mathbb{H}^{n}$ with $\Pi$ as its second fundamental form.  
Since $\Pi\geq0$, by the remark in Section 5 of Do Carmo and Warner \cite{dCW70}, $ \Phi $ is  an embedding and $\Phi (\Sigma)$ is   a convex
hypersurface in $\mathbb{H}^n$.
Apply Theorem \ref{thm-hyperbolic} to $(\Omega, g)$ and the embedding $\Phi$, we conclude that 
$g$ has constant sectional curvature $-1$ everywhere on $\Omega$. 
Now let $D$ be the region bounded by $ \Phi (\Sigma) $ in $\mathbb{H}^n$. We glue $\Omega$ and 
$\mathbb{H}^n \setminus D$ along the boundary via the isometric embedding $\Phi$ to get a complete  manifold  $(\widetilde{{M}}, \tilde{g})$.
The fact    $\Sigma$ has the same second fundamental form in $(\Omega, g)$ and $\mathbb{H}^n$ 
and  both $(\Omega, g)$ and $\mathbb{H}^n$ have constant sectional curvature $-1$ imply that 
$\tilde{g}$ is smooth across $\Sigma$ in $\widetilde{{M}}$. Hence, $(\widetilde{M}, \tilde{g})$ is a complete, hyperbolic  manifold 
which, outside a compact set,  is isometric to $\mathbb{H}^n$ minus a ball.
We conclude that $ (\widetilde{M}, \tilde{g})$ is isometric to $\mathbb{H}^n$ 
and  $(\Omega, g)$ is isometric to a domain in $\mathbb{H}^n$.  
This final claim can be seen, for instance, by the following:

\begin{prop}\label{H^nK}
Let $(M,g)$ be a complete, $n$-dimensional Riemannian manifold with  $\Ric\geq-\left(  n-1\right) g $. Suppose that there exists a compact set $K\subset M$ s.t. $M \setminus K$ is isometric to $\mathbb{H}^n \setminus B$
where $B$ is a geometric ball.
Then $M$ is isometric to $\mathbb{H}^n$.
\end{prop}
The Euclidean version of the result is well known (e.g. it appears as an exercise in \cite{P06} several times) . 
The hyperbolic case can be proved by similar methods. For lack of an exact reference, we outline a proof using Busemann functions.
The main idea comes from Cai-Galloway \cite{CG00}.
We use the upper space model $\mathbb{H}^n=\{x\in\R^n:x_n>0\}$. Without loss of generality we take $o=(0,\ldots,0,1)$. For a sequence $\e_k\rightarrow 0$, let $S_k\subset M$ be 
the hypersurface corresponding to $x_n=\e_k$ and $q_k$ the point corresponding to $(0,\ldots,0,1/\e_k)$. Let $p_k$ be the point on $S_k$ closest to $q_k$ and
$\gamma_k:[-a_k,b_k]\rightarrow M$ be a minimizing geodesic from $p_k$ to $q_k$ s.t. $\gamma_k(0)\in K$ (it is easy to see that any minimizing geodesics from $p_k$ to $q_k$ must intersect $K$).  Passing to a subsequence $\e_k\rightarrow 0$ if necessary, we can assume that $\gamma_k(0)\rightarrow \bar{o}$ and $\gamma_k$ converges to a geodesic line $\gamma:\R\rightarrow M$.
We consider the following generalized Busemann function
$$\beta(x)=\lim_{k\rightarrow \infty}d(\bar{o},S_k)-d(x,S_k).$$
Then we have 

{\bf Claim:} $\Delta \beta\geq n$ in the support sense.

The crucial fact is that $S_k$ has constant mean curvature $H=n-1$. The argument is the same as in Cai-Galloway \cite{CG00}.

We also have the standard Busemann function $b$ associated with the ray $\gamma|_{[0,\infty)}$ defined by
$b(x)=\lim_{k\rightarrow \infty}s-d(x,\gamma(s))$. It is known that $\Delta b\geq -n$ in the support sense. 
The rest of the proof is similar to Cai-Galloway \cite{CG00} or the proof of the Cheeger-Gromoll splitting 
theorem (cf. e.g. \cite{P06}). We have $\Delta(b+\beta)\geq 0$. By the triangle inequality one can show 
$b+\beta\leq 0$. On the other hand $b+\beta=0$ along $\gamma$. Therefore by the strong maximum principle $b+\beta=0$.
Then $\beta=-b$ and it is a smooth function with $|\nabla \beta|=1$. By the Bochner formula one can show that $\nabla^2\beta=g-d\beta\otimes d\beta$.
From this identity one can show that $M$ is isometric to the warped product $(\R\times S^{n-1}, dt^2+e^{2t} h)$, where $(S,h)$ is a flat Riemannian manifold. 
It is then clear that $(S,h)$ must be the standard $\R^{n-1}$. This finishes the proof of Proposition \ref{H^nK}.

\end{document}